\documentclass[11pt]{amsart}
\usepackage{latexsym,amssymb,amsmath}
\usepackage[colorlinks=true, linkcolor=blue, anchorcolor=black, citecolor=blue, filecolor=blue, menucolor= blue, urlcolor=black]{hyperref}
\usepackage{graphicx}
\usepackage{float}
\usepackage{faktor}
\usepackage[all]{xy}
\usepackage{tikz}
\usepackage{minted}
\usepackage{amsthm}
\usepackage[numbers,sort&compress]{natbib}
\usetikzlibrary{shapes.geometric, fit, positioning}

\newtheoremstyle{myplain}
  {6pt}
  {6pt}
  {\itshape}
  {}
  {\bfseries}
  {.}
  { }
  {\thmname{#1}\thmnumber{ #2}\thmnote{. {\normalfont #3}}}

\newtheoremstyle{mydefinition}
  {6pt}
  {6pt}
  {\normalfont}
  {}
  {\bfseries}
  {.}
  { }
  {\thmname{#1}\thmnumber{ #2}\thmnote{. {\normalfont #3}}}

\theoremstyle{myplain}
\newtheorem{theorem}{Theorem}[section]
\newtheorem{lemma}[theorem]{Lemma}
\newtheorem{proposition}[theorem]{Proposition}
\newtheorem{corollary}[theorem]{Corollary}
\newtheorem{conjecture}[theorem]{Conjecture}

\theoremstyle{mydefinition}
\newtheorem{definition}[theorem]{Definition}

\newtheorem{remark}[theorem]{Remark}
\newtheorem{example}[theorem]{Example}

\usepackage{booktabs}
\usepackage{tabularx}
\usepackage{array}

\newcommand{\set}[1]{\left\{ #1 \right\}}

\newcommand{\Ass}{\operatorname{Ass}}

\newcommand{\PN}{\mathtt{P}_{\mathtt{N}}}
\newcommand{\Nt}{\mathtt{N}}
\newcommand{\Nc}{\mathcal{N}}
\newcommand{\vt}{\mathtt{v}}
\newcommand{\regt}{\mathtt{reg}}

\newif\ifcomment
\commentfalse % Uncomment to not show any comments

\begin{document}

%%%%%%%%%%%%%%%%%%%%%%%%%%%%%%%%%%%%%%%%%%%%%%%%%%%%%%%%%%%%%%%%%%%%%

\title[Private neighbors, perfect codes, $\vt$-number, and closed neighborhood ideals]{Private neighbors, perfect codes, and their relation with the $\vt$-number of closed neighborhood ideals}

\author[Jaramillo-Velez]{Delio Jaramillo-Velez}
\address{(Jaramillo-Velez) 
Dpto. Matemáticas, Estadística e Investigación Operativa.
Instituto Universitario de Matemáticas y Aplicaciones (IMAULL). 
Universidad de La Laguna. Apartado de Correos 456. 38200 La Laguna, Tenerife, Spain,
and\\
Department 
of Mathematics\\
Virginia Tech\\
Blacksburg, VA, USA\\
}
\email{djaramil@ull.es, delio@vt.edu}

\author[L\'opez]{Hiram H. L\'opez}
\address{(L\'opez) Department of Mathematics\\
Virginia Tech\\
Blacksburg, VA, USA}
\email{hhlopez@vt.edu}

\author[San-Jos\'e]{Rodrigo San-Jos\'e}
\address{(San-Jos\'e) Department of Mathematics\\
Virginia Tech\\
Blacksburg, VA, USA}
\email{rsanjose@vt.edu}

\thanks{Delio Jaramillo-Velez was partially supported by the NSF grant DMS-2401558. Hiram H. L\'opez was partially supported by the Commonwealth Cyber Initiative and by the NSF grants DMS-2401558 and DMS-2502705. Rodrigo San-José was partially supported by the Commonwealth Cyber Initiative, the NSF grant DMS-2401558, and by Grant PID2022-137283NB-C22 funded by MICIU/AEI/10.13039/501100011033 and by ERDF/EU}

\maketitle

\begin{abstract}
In this work, we investigate the connections between dominating sets, private neighbors, and perfect codes in graphs, and their relationships with commutative algebra. In particular, we estimate the $\vt$-number of closed neighborhood ideals in terms of minimal dominating sets and private neighbors. We show how the $\vt$-number is related to other graph invariants, such as the cover number, domination number, and matching number. Moreover, we explore the relation with the Castelnuovo-Mumford regularity, proving that the $\vt$-number is a lower bound for the regularity of bipartite, very well-covered, and chordal graphs. Finally, drawing from the relation between efficient dominating set and perfect codes, we use the redundancy of Hamming codes to present lower and upper bounds for the $\vt$-number of some special family of graphs.
\end{abstract}

%%% section 1 introduction

\section{Introduction}
A dominating set in a graph is a subset of vertices $D$ such that any other vertex of the graph has a neighbor in $D$. A large part of the interest in dominating sets comes from its wide range of connections across many areas, including chemistry, computer communication networks, facility location, social networks, surveying, monitoring electrical power networks, genetics, coding theory, and several other branches of mathematics; see~\cite{domination_book}. As an introductory example of these connections, consider a graph representing a communication network in which information travels along the edges. Then, assume that transmitters are placed on a set of vertices $D$ and that $u$ is an external private neighbor of $D$, meaning that $u$ has only one neighbor $v$ in $D$. Then, removing $v$ from $D$ would directly affect the network's communication because no other transmitter in $D \setminus \{v\}$ would reach $u$. In this sense, we say that $D$ minimally dominates the communication of the network~\cite{generalized_ramsey}. The previous example illustrates the importance of dominating sets and also introduces the relevant concept of private neighbors. These vertices are such that make the set $D$ minimal with respect to the property of domination~\cite{intro_private_nei,privatecube,prmaximizing}. In this work, we show relationships between dominating sets, private neighbors, perfect codes of a graph, and commutative algebra. We do this by analyzing the $\vt$-number of closed neighborhood ideals.

The Vasconcelos number, or just $\vt$-number, is an algebraic invariant of a graded ideal $I$ in a polynomial ring $S:=\mathbb{K}[t_1,\dots,t_n]$ over a field $\mathbb{K}$, defined as
\begin{equation}\label{vasconcelos}
\vt(I) := \min \left\{ d \geq 0 \;\middle|\; \exists \ f \in S_d \text{ and } \mathfrak{p} \in \Ass(I) \text{ with } (I:f)=\mathfrak{p} \right\},
\end{equation}
where $S_d$ denotes the $d$-th graded component of $S$, $\Ass(I)$ denotes the set of associated primes of $I$, and $(I:f)$ is the colon ideal.
The $\vt$-number was introduced by Cooper, Seceleanu, Toh\u{a}neanu, Vaz Pinto, and Villarreal in connection with the study of the asymptotic behavior of the minimum distance of projective Reed–Muller-type codes~\cite{intro-vnumber, Villarreal2026}. This algebraic invariant is tied to indicator functions, which arise in coding theory~\cite{dualcode,sorensenprojective}, Cayley–Bacharach schemes~\cite{geramitacayley}, and interpolation problems~\cite{kreuzercomputational}.

The $\vt$-number has served as a rich source of algebraic interpretations of graph invariants thanks to ideals associated with graphs. One of the most well-known of these ideals is the edge ideal $I(G)$, which is essentially generated by the edges of the graph $G$; see~\cite{introedg}. Here, the $\vt$-number of $I(G)$ corresponds to the minimal cardinality of an independence set of $G$ such that its neighbors are a minimal vertex cover~\cite{vnumberedge, induced}. Such a characterization provides algebraic conditions to derive conclusions about the graph~\cite{vnumberedge}; for example, an algebraic classification of $W_2$ graphs. For the case of binomial edge ideals, the $\vt$-number at the first minimal prime corresponds with the connected domination number~\cite{v-bino2,v-bino1}, giving thus a relation between the $\vt$-number and the concept of domination.

Motivated by the relations between commutative algebra and graph theory, in this paper, we establish a combinatorial expression of the $\vt$-number of closed neighborhood ideals; see Theorem~\ref{t:v_number_formula}. The closed neighborhood ideal of a graph $G$ is generated by the closed neighborhoods of the vertices of $G$. These ideals can be viewed as a generalization of edge ideals, but instead of considering only one neighbor per vertex, we consider its entire neighborhood~\cite{def-NG}. Recently, neighborhood ideals have attracted considerable attention, and researchers have presented several combinatorial-algebraic results that include an expression for the primary decomposition, which depends on the minimal dominating sets~\cite{irred-decom-NG}. Moreover, there has been considerable interest in the combinatorial estimation of homological invariants of neighborhood ideals, including the Castelnuovo–Mumford regularity and the projective dimension~\cite{homologicalinvar,castelnuovomatch,BM_reso_closed}.

In Section~\ref{sec:prel}, we introduce notation and preliminary results regarding the main subjects of this work: graph theory, commutative algebra, and perfect codes. We refer the reader to Table~\ref{tab:notation} for a summary of the notation we use in this work.

Section~\ref{sec:v-number} is devoted to presenting a combinatorial expression for the $\vt$-number of closed neighborhood ideals, which corresponds to Theorem~\ref{t:v_number_formula}. This expression is purely in terms of minimal dominating sets and their associated private neighbors.

In Section~\ref{sec:invariants}, we present relations between the $\vt$-number and several other graph invariants such as the cover number, the domination number, and the matching number; see Theorem~\ref{thm:cover-v}. We address the relation between the $\vt$-number and the Castelnuovo-Mumford regularity. We show that the $\vt$-number is a lower bound for several families of graphs, and conjecture the general case; see Theorem~\ref{thm:vreg} and Conjecture~\ref{conj:vnumber-reg}. This is motivated by the now classical case of edge ideals, where the $\vt$-number can be arbitrarily greater than the regularity~\cite{civan_v_number}.

In Section~\ref{sec:perfect-codes}, we consider perfect codes in the vector space $\mathbb{F}_q^m$, where $\mathbb{F}_q$ is a finite field. A code $C \subseteq \mathbb{F}_q^m$ is $e$-perfect if for every element $v \in \mathbb{F}_q^m$ there is a unique element $w$ in $C$ such that $w$ and $v$ differ in at most $e$ entries. By representing the vector space $\mathbb{F}_q^m$ as a specific graph $\Gamma(m,q)$, the problem of the existence of $e$-perfect codes corresponds with finding special domination sets, called efficient dominating sets, with the property that the neighborhoods of their elements form a partition of the graph~\cite{prmaximizing,perfectin}. The problem of the existence of $e$-perfect codes was solved by Tiet{\"a}v{\"a}inen, who proved that the only nontrivial perfect codes are the $1$-error-correcting Hamming codes together with two exceptional codes first discovered by Golay~\cite{existing_perfe}. We connect the $\vt$-number to coding theory by presenting lower and upper bounds for the $\vt$-number of the closed neighborhood ideal $\Nc_{\Gamma\left(\frac{q^r-1}{q-1}, q \right)}$ in terms of the redundancy of the existent Hamming code; see Theorem~\ref{thm:Hammingdomination}.

In Appendix~\ref{sec:computer_prog}, we provide the Sage code used for implementations of the closed neighborhood ideals, $\vt$-number formula, and computations of the Castelnuovo-Mumford regularity~\cite{sagemath}.

\section{Preliminaries}\label{sec:prel}
In this section, we present terminology and preliminary results. We divide the section into three parts: graph theory, commutative algebra, and coding theory. Additionally, Table~\ref{tab:notation} summarizes all of the notation. 

\subsection{Graph theory}
A \textit{graph} $G$ is a pair $(V(G), E(G))$, where $V(G)$ is a finite set and \[E(G) = \{\{u,v\} \mid u,v \in V(G) \}.\]
The elements of $V(G)$ and $E(G)$ are called vertices and edges of $G$, respectively. We consider only \textit{simple graphs}, i.e., graphs without multiple edges or loops. We say that a graph is \textit{connected} if for every pair of vertices there is a path between them. A graph $H$ is a subgraph of $G$ if $V(H) \subseteq V(G)$ and $E(H) \subseteq E(G)$. A subgraph $H$ of $G$ is called \textit{induced} if
\[
E(H) = \{ \{u,v\} \in E(G) \mid u,v \in V(H) \}.
\]

For an edge $e = \{u,v\} \in E(G)$, we say that $e$ is incident to $u$ and $v$, $u$ is adjacent to $v$, or $u$ is a neighbor of $v$.

Let $D \subseteq V(G)$ be a set of vertices and $M \subseteq E(G)$ a set of edges of a graph $G$.
\begin{itemize}
    \item The set $D$ is a \emph{vertex cover} of $G$ if every edge of $G$ is incident with at least one vertex in $D$. A vertex cover is \emph{minimal} if it is minimal with respect to inclusion. The \emph{vertex cover number} of $G$, denoted by $\tau(G)$, is the minimum cardinality of a vertex cover of $G$. 
    \item The set $D$ is an \emph{independent set} of $G$ if no two vertices in $D$ are adjacent. An independent set is \emph{maximal} if it is maximal with respect to inclusion. The \emph{independence number} of $G$, denoted by $i(G)$, is the maximum cardinality of an independent set. 
    \item The set $M$ is a \emph{matching} of $G$ if no two edges in $M$ share a vertex. The \emph{matching number} of $G$, denoted by $a(G)$, is the maximum cardinality of a matching of $G$. 
\end{itemize}

A graph $G$ is called chordal if any cycle of length $m\geq 4$ in $G$ has a chord, which is an edge that is not part of the cycle but connects two vertices of the cycle. The graph $G$ is called well-covered if all maximal independent sets have the same cardinality, and it is called very well-covered if it is well-covered, without isolated vertices, and such that the size of any minimal vertex cover (maximal independent set) is half of the number of vertices of the graph.

For a vertex $v \in V(G)$, we define:
\begin{itemize}
    \item The \textit{neighborhood} of $v$, denoted by $\mathtt{N}(v)$, is the set of all neighbors of $v$.
    \item The \textit{closed neighborhood} of $v$ is given by
    $\mathtt{N}[v] := \mathtt{N}(v) \cup \{v\}.$
\end{itemize}
Similarly, for a set of vertices $D \subseteq V(G)$, we define:
\begin{itemize}
    \item The \textit{neighborhood} of $D$ is
    $\mathtt{N}(D) := \bigcup_{v \in D} \mathtt{N}(v).$
    \item The \textit{closed neighborhood} of $D$ is
    $\mathtt{N}[D] := \mathtt{N}(D) \cup D.$
\end{itemize}

\subsubsection{Private neighbors}
The concept of private neighbors in graphs was introduced by Cockayne, Hedetniemi, and Miller in 1978~\cite{intro_private_nei}. Let $v \in V(G)$ be a vertex and $D \subset V(G)$ a set of vertices in a graph $G$. 

Assume that $v$ is a \emph{private neighbor} of $D$, meaning that $|\mathtt{N}[v] \cap D| = 1$.
\begin{itemize}
        \item If $v \notin D$, then $v$ is called an \emph{external private neighbor} of $D$.
        \item If $v \in D$ and $\deg(v)=0$ in the induced graph $G[D]$, then $v$ is called a \emph{self-private neighbor} of $D$; see~\cite{prmaximizing}.
        \item If $v \in D$ and $\deg(v)=1$ in the induced graph $G[D]$, then $v$ is called an \emph{internal private neighbor} of $D$; see~\cite{prmaximizing}.
\end{itemize}
We denote the set of all \textit{external and self-private neighbors} of the set $D$ by $\mathtt{P}_{\mathtt{N}}(D)$. 

\begin{example}
In the graph display in Figure~\ref{fig:private_nei_exa}, we consider the set of vertices $D=\{v_1,v_4,v_6\}$. The vertex $v_1$ is a self-private neighbor of $D$, the vertices $v_5, v_7$ are external private neighbors of $D$, and the vertices $v_4, v_6$ are internal private neighbors of $D$. We can also see that
$$
\PN(D)=\{v_1,v_5,v_7\}.
$$
\begin{figure}[h]
    \centering
     \begin{tikzpicture}[scale=1.2,
    every node/.style={font=\small},
    dot/.style={circle, fill=black, inner sep=2.5pt},
    bigdot/.style={circle, fill=blue, inner sep=4pt}
]

% Vertices
\node[bigdot, label=left:$v_1$]  (v1) at (0,0) {};
\node[dot,    label=below:$v_2$] (v2) at (1.6,-1.5) {};
\node[dot,    label=above:$v_3$] (v3) at (2.4,1.5) {};
\node[bigdot, label=below:$v_4$] (v4) at (4.3,0.0) {};
\node[dot,    label=above:$v_5$] (v5) at (5.9,1.5) {};
\node[bigdot, label=below:$v_6$] (v6) at (6.9,0.0) {};
\node[dot,    label=above:$v_7$] (v7) at (8.6,1.0) {};

% Edges
\draw[thick] (v1) -- (v3);
\draw[thick] (v3) -- (v4);
\draw[thick] (v4) -- (v5);
\draw[thick] (v4) -- (v6);
\draw[thick] (v6) -- (v7);
\draw[thick] (v1) -- (v2);
\draw[thick] (v2) -- (v6);
\draw[thick] (v3) -- (v5);

\end{tikzpicture}
    \caption{Private neighbors of the set $D=\{v_1,v_4,v_6\}$.}
    \label{fig:private_nei_exa}
\end{figure}
   
\end{example}

\subsubsection{Dominating and irredundant sets}
We now introduce the concept of an irredundant set of a graph, which is related to private neighbors.

Let $D$ and $U$ be two sets of vertices of a graph $G$. We say that $D$ \emph{dominates} $U$ if 
\[
U \subseteq \Nt[D].
\]
The set $D$ is a \emph{dominating set} of $G$ if $\Nt[D] = V(G)$. The \emph{domination number} of $G$, denoted by $\gamma(G)$, is the minimum cardinality of a dominating set of $G$. A dominating set is called \textit{minimal} if it does not properly contain another dominating set. Note that if a dominating set $D$ is minimal, then each vertex $u$ of $D$ has a private neighbor, meaning that $|\Nt[u]\cap \PN(D)|\geq 1$. Therefore, if we remove any vertex from $D$, the resulting set is no longer a dominating set. A reference for dominating sets is~\cite{domination_book}.

\begin{definition}
A set of vertices $D \subset V(G)$ is an \emph{irredundant} set of $G$ if for every $u \in D$ there is $v \in V(G)$  such that 
\[
\Nt[v] \cap D = \{u\}.
\]
In other words, a set $D$ is irredundant if and only if $|\Nt[u] \cap \PN(D)|\geq 1$ for every $u\in D$. An irredundant set $D$ is \emph{maximal} if it is not properly contained in another irredundant set.
\end{definition}  

The following result, which can be found in~\cite[Cor. 10.37]{chartrand_graphs_digraphs}, relates minimal dominating sets with maximal independent sets.
\begin{proposition} \label{p:maximal_independent_is_minimal_dominating}
Every maximal independent set is a minimal dominating set.
\end{proposition}
\begin{remark}
As a consequence of Proposition~\ref{p:maximal_independent_is_minimal_dominating}, we always have $\gamma(G)\leq i(G)$.     
\end{remark}

\subsection{Commutative algebra}
We denote the polynomial ring  over a field $\mathbb{K}$ with the standard grading by 
$$
S := \mathbb{K}[t_1, \ldots, t_n] = \bigoplus_{d=0}^{\infty} S_d.
$$
Let $I$ be a graded ideal of $S$. A prime ideal $\mathfrak{p}$ of $S$ is an \emph{associated prime} of $S/I$ if
$$
(I : f) := \mathfrak{p}
$$
for some $f \in S_d$, where
$$
(I : f) = \{ g \in S \mid gf \in I \}
$$
is the colon ideal between $I$ and $f$. The set of associated primes of $S/I$ is denoted by $\operatorname{Ass}(I)$. We recall that the $\mathtt{v}$-\emph{number} of $I$ is defined by
$$
\vt(I) := \min \{ d \geq 0 \mid \exists\, f \in S_d \text{ and } \mathfrak{p} \in \operatorname{Ass}(I) \text{ with } (I : f) = \mathfrak{p} \}.
$$

The $\vt$-number of $I$ can be also defined locally at each associated prime $\mathfrak{p}$ of $I$ by
$$
\vt_{\mathfrak{p}}(I) := \min \{ d \geq 0 \mid \exists\, f \in S_d \text{ with } (I : f) = \mathfrak{p} \}.
$$

In this work, we are interested in the relation between the $\vt$-number and the Castelnuovo-Mumford regularity, which is defined as follows. Consider the minimal graded free resolution of $S/I$ as an $S$-module:
$$
0 \to \bigoplus_j S(-j)^{b_{g,j}}
\to \cdots \to
\bigoplus_j S(-j)^{b_{1,j}}
\to S \to S/I \to 0.
$$
The \emph{Castelnuovo--Mumford regularity} of $S/I$, or simply the \emph{regularity} of $S/I$, is defined as
$$
\regt(S/I) := \max\{j-i \mid b_{i,j} \neq 0\}.
$$
The integer $g$, denoted by $\mathtt{pd}(S/I)$, is the \emph{projective dimension} of $S/I$; see~\cite{intro-vnumber, Villarreal2026}.

For the rest of the manuscript, we identify the set of variables $\{t_1, \dots, t_n\}$ with the set of vertices $V(G)$ of a graph $G$. We use the notation $\langle D\rangle$ to indicate the ideal generated by the variables or vertices in $D\subseteq V(G)$. The correspondence between variables and vertices allows us to define families of squarefree monomial ideals parametrized by subsets of vertices, such as edge ideals~\cite{edgeideal}, which is defined as follows. The \emph{edge ideal} of a graph $G$ is denoted and defined by 
$$
I(G):=\langle t_it_j\:|\: \{t_i,t_j\}\in E(G)\rangle.
$$
We are interested in the family of closed neighborhood ideals, which naturally extends the notion of edge ideal from a combinatorial perspective~\cite{def-NG}. For a set of vertices $D\subseteq V(G)$, the square-free monomial parametrized by $D$ is denoted by
\[
\mathbf{t}_D:=\prod_{t_j \in D} t_j.
\]
The \emph{closed neighborhood ideal} of $G$, denoted by $\Nc_G$, is defined as
$$
\Nc_G:=\langle \mathbf{t}_{\Nt[t_i]}\:|\: t_i\in V(G)\rangle.
$$
\begin{remark}
A generator of an edge ideal is the product of a vertex with one of its neighbors, while a generator of a closed neighborhood ideal is the product of a vertex with all of its neighbors.
\end{remark}

\subsection{Coding theory}

Let $\mathbb{F}_q$ be the finite field of order $q$, where $q = p^r$ is a prime power. We say that $C$ is an $[m,k,\delta]$ code if $C$ is a linear code over $\mathbb{F}_q$ of length $m$, dimension $k$, and minimum distance $\delta$. Observe that the code $C$ is a $k$-dimensional subspace of the vector space $\mathbb{F}_q^m$ and $|C| = q^k$. The code $C$ can correct up to $t$ errors~\cite{huffman_codes}, where
\begin{equation}\label{26.05.27}
t := \left\lfloor \frac{\delta-1}{2} \right\rfloor.    
\end{equation}

A Hamming ball $S_\rho(x)$ of radius $\rho$ with center at the vector
\[
x = (x_1, \dots, x_m) \in \mathbb{F}_q^m
\]
is defined as the set
\[
S_\rho(x) := \{y \in \mathbb{F}_q^m : \delta(x,y) \le \rho\},
\]
where $\delta(x,y)$ denotes the Hamming distance between the vectors $x$ and $y$; this is the number of coordinates in which $x$ and $y$ differ. The well-known Hamming bound can be stated as follows.
\begin{theorem} [\cite{huffman_codes}] \label{t:sphere_packing}
Let $C $ be an $[m,k,\delta]$ code and define $t$ as in Equation~(\ref{26.05.27}). Then,
$$
|C| \sum_{i=0}^{t} \binom{m}{i}(q-1)^i \le q^m.
$$
\end{theorem}
If $C$ achieves the previous bound, we say that $C$ is \emph{perfect}. A perfect code can be understood as follows. Around each codeword $c\in C$, we draw a Hamming ball $S_t(c)$ containing all vectors whose Hamming distance from $c$ is at most $t$. The elements in the ball represent all possible received codewords with at most $t$ errors that can be efficiently decoded to $c$. A code is perfect if these balls do not overlap and the union of the balls covers the entire space $\mathbb{F}_q^m$. In other words, every possible received codeword lies in exactly one Hamming ball, meaning the space is filled perfectly with no gaps and no overlaps.

%%% Notation table
\begin{table}[ht]
\centering
\renewcommand{\arraystretch}{1.3}
\begin{tabularx}{\textwidth}{>{\raggedright\arraybackslash}p{0.25\textwidth} X}
\toprule
\textbf{Notation} & \textbf{Description} \\
\midrule

%%%%%%% graph  notation

$G$ & A simple graph\\

$\mathtt{N}(v)$ & Open neighborhood of the vertex $v$ \\

$\mathtt{N}[v]$ & Closed neighborhood of the vertex $v$ \\

$D^c$ & Complement of the set of vertices $D$ \\

$\mathtt{P}_{\mathtt{N}}(D)$ & Set of external and self private neighbors of $D$\\

$\gamma(G)$ & Domination number of the graph $G$\\

$\tau(G)$ & Cover number of the graph $G$\\

$i(G)$ & Independence number of the graph $G$\\

$a(G)$ & Matching number of the graph $G$\\

%%%%%%% commut-alg notation

$S$ & Polynomial ring in $n$ variables\\

$(I:f)$ & Colon ideal of $I$ and $f$ \\

$\operatorname{Ass}(I)$ & Assosiated primes of $I$ \\

$\vt(I)$ & The $\vt$-number of $I$ \\

$\mathtt{reg}(S/I)$ & Castelnuovo--Mumford regularity of $S/I$ \\

$\mathbf{t}_A$ & Square-free monomial parametrized by the variables in $A$ \\

$\langle D\rangle$ & Ideal generated by the elements in $D$ \\
  
$\Nc_G$ & Closed neighborhood ideal associated to $G$\\

$\mathbf{t}_A|\mathbf{t}_B$ & The monomial $\mathbf{t}_A$ divides the monomial $\mathbf{t}_B$\\

%%%%%%% coding theory notation

$\mathcal{H}_q(r)$ & Hamming code with redundancy $r$\\
$\delta(x,y)$  &  Hamming distance between $x$, and $y$\\

$\Gamma(m,q)$ & Hamming graph representation of $\mathbb{F}_q^{m}$\\

\bottomrule
\end{tabularx}
\caption{Table of Notation}
\label{tab:notation}
\end{table}

%%% section 3 the v-number 

\section{The $\vt$-number of neighborhood ideals}\label{sec:v-number}
In this section, we present a formula for the $\vt$-number of a closed neighborhood ideal. To this end, we employ the notion of private neighbors and leverage the following description of the minimal primes of closed neighborhood ideals.

\begin{proposition}[\cite{irred-decom-NG}]\label{propo:irred-decom-NG}
The closed neighborhood ideal $\Nc_G \subseteq S$ has the following irreducible decompositions:
\[
\Nc_G = \bigcap_{\substack{D' \text{ dominating set}}} \langle D'\rangle
    = \bigcap_{\substack{ D' \text{ minimal dominating set}}} \langle D'\rangle.
\]
Furthermore, the second intersection is an irredundant decomposition.
\end{proposition}

The following relation between dominating and irredundant sets is particularly useful for our proposed formula for the $\vt$-number of closed neighborhood ideals. We add a proof of this result for completeness and clarity. 

\begin{proposition}[\cite{intro_private_nei}]\label{mini-irre}
If $D$ is minimal dominating, then $D$ is irredundant. 
\end{proposition}

\begin{proof}
We reason by contradiction. We assume the existence of a minimal dominating set $D$ that is not irredundant. Then, there exists an element $u\in D$ such that for all $v\in V(G)$, 
$$u\not\in \Nt[v] \quad \hbox{ or } \quad \Nt[v]\setminus\{u\}\not\subseteq D^c.$$ We now show that this implies that $B:=
D\setminus\{u\}$ is a dominating set, i.e., for every $x\in V(G)$, there is $y\in B$ that is adjacent to $x$.
\begin{itemize}
    \item Assume $x=u$. There is $y\in \Nt[u]\setminus\{u\}$ and $y\not\in D^c$. Thus, $u$ is adjacent to an element $y$ in $B$.
    \item Assume that $x\not= u$ and $x\not\in D$. Since $D$ is a dominating set, there is $w\in D$ with $x\in \Nt(w)$. If $w\not= u$, then $x$ is adjacent to a vertex in $B$. If $w=u$, by the contradiction assumption, we obtain that $\Nt[x]\setminus\{u\}\not\subseteq D^c$. Thus, there exists $y \in \Nt[x]\setminus\{u\}$ and $y\in D$. This implies that $x$ is adjacent to $y\in B$, and it shows that $B$ is a dominating set.
\end{itemize}
In both cases, we obtain a contradiction with the fact that $D$ is a minimal dominating set. 
\end{proof}

\begin{remark}
For monomial ideals, the $\vt$-number satisfies an additivity property analogous to that of Castelnuovo–Mumford regularity. Specifically, if $I_1 \subset R_1$ and $I_2 \subset R_2$ are monomial ideals in polynomial rings $R_1$ and $R_2$ whose sets of variables are disjoint, then
$$\vt(I_1+I_2)=\vt(I_1)+\vt(I_2),$$
where $I_1+I_2$ is viewed as an ideal in the polynomial ring generated by the variables of both $R_1$ and $R_2$~\cite{vnumber_monomials, additivity_reg}. This observation shows that, for our analysis of the v-number of closed neighborhood ideals and its connection with Castelnuovo–Mumford regularity, it suffices to consider connected simple graphs.
\end{remark}

\begin{remark}\label{rem:colon_dom}
Let $U$ and $D$ be sets of vertices such that with $(\Nc_G: \mathbf{t}_U)=\langle D\rangle$. We have that for all vertices $w\in D$, there exists a vertex $v$ such that $\Nt[v]\subseteq U\cup\{w\}$ and $\Nt[v]\not\subseteq U$ because $(\Nc_G: \mathbf{t}_U)\neq S$. This implies that $U\cap D=\emptyset$. 
\end{remark}

\begin{lemma}\label{lemma:vnumber}
    Let $\Nc_G$ be the closed neighborhood ideal of a simple connected graph $G$. The following holds.
    \begin{itemize}
        \item[(a)] For a minimal dominating set $D$ and a set $U\subseteq \mathtt{P}_{\mathtt{N}}(D)$ that dominates $D$, we have that
        \[(\Nc_G : \mathbf{t}_{\Nt[U]\setminus D}) = \langle D\rangle.
        \]
        \item[(b)] Let $D$ be a minimal dominating set and $f\in S_d$ such that
        $(\Nc_G : f) = \langle D\rangle$. There  is a set of vertices $U\subseteq \PN(D)$ that dominates $D$ such that
        \[(\Nc_G : \mathbf{t}_{\Nt[U]\setminus D}) = \langle D\rangle \quad \text{ and } \quad |\Nt[U]\setminus D|\leq \mathtt{deg}(f).\] 
    \end{itemize}
\end{lemma}

\begin{proof}
{(a)} We prove first $(\Nc_G : \mathbf{t}_{\Nt[U]\setminus D}) \subseteq \langle D\rangle$. Let $\mathbf{t}^{\mathbf{c}} = t_1^{c_1}\cdots t_n^{c_n}$ be an element in the colon ideal $(\Nc_G:\mathbf{t}_{\Nt[U]\setminus D})$. We denote by $\mathtt{supp}(\mathbf{t}^{\mathbf{c}})$ the support of $\mathbf{t}^{\mathbf{c}}$, which is the set of variables with a positive exponent. There exists a vertex $v$ such that  $$
        \mathbf{t}_{\Nt[v]} | \mathbf{t}^{\mathbf{c}}\mathbf{t}_{\Nt[U]\setminus D}.
        $$ 
        Given that $D$ is a dominating set, there exists $t_j\in D$ with 
        $$
        t_j\in \Nt[v]\subseteq\mathtt{supp}(\mathbf{t}^{\mathbf{c}}\mathbf{t}_{\Nt[U]\setminus D})=\mathtt{supp}(\mathbf{t}^{\mathbf{c}})\cup (\Nt[U]\setminus D).
        $$ 
        Thus, $t_j\in \mathtt{supp}(\mathbf{t}^{\mathbf{c}})$, which implies that $\mathbf{t}^{\mathbf{c}}\in \langle D\rangle$.
        
        We now check $(\Nc_G : \mathbf{t}_{\Nt[U]\setminus D}) \supseteq \langle D\rangle$. Let $t_j$ be an element in $D$. Since $U \subseteq \mathtt{P}_{\mathtt{N}}(D)$ and $U$ dominates $D$, there is $v_j\in U$ such that $\Nt[v_j]\cap D=\{t_j\}$. This implies that 
        $$
        \Nt[v_j] = \mathtt{supp}(t_j\mathbf{t}_{\Nt[v_j]\setminus D}) \subseteq \mathtt{supp}(t_j\mathbf{t}_{\Nt[U] \setminus D}). 
        $$
        Therefore, $t_j\in (\Nc_G:\mathbf{t}_{\Nt[U]\setminus D})$ and $\langle D\rangle\subseteq (\Nc_G:\mathbf{t}_{\Nt[U]\setminus D})$.

{(b)} Write the polynomial $f$ as sum of monomials:
\[
f = \sum_{i=1}^{r} \lambda_i \mathbf{t}^{\mathbf{c}_i},
\]
where $0 \neq \lambda_i \in \mathbb{K}$ and $\mathbf{t}^{\mathbf{c}_i}:=t_1^{c_{i1}} \cdots t_n^{c_{in}} \in S_d$ for all $i$. Then, we obtain 
\[
(\Nc_G : f) = \bigcap_{i=1}^{r} (\Nc_G : \mathbf{t}^{\mathbf{c}_i}) = \langle D\rangle,
\]
and consequently $(\Nc_G : \mathbf{t}^{\mathbf{c}_k}) = \langle D\rangle$ for some $k$.

Given that $\Nc_G$ is a squarefree monomial, we may  assume that $c_{kj} \in \{0,1\}$ for all $j$, and we consider $A := \mathtt{supp}(\mathbf{t}^{\mathbf{c}_k}).$ Therefore, 
$$(N_G : \mathbf{t}_{A}) = \langle D\rangle.$$ 
Moreover, we can assume $A\cap D=\emptyset$ by Remark \ref{rem:colon_dom}. For every $t_j\in D$, there exists $v_j\in V(G)$ such that 
$$
\mathbf{t}_{\Nt[v_j]} |t_j\mathbf{t}_{A}.
$$
This implies that $\Nt[v_j]\setminus\{t_j\}\subseteq A$. We consider the set formed by all the $v_j$'s, this is $U:=\{v_1,\dots,v_{|D|}\}$. Notice that $D\subseteq \Nt[U]$ and $|\Nt[v_j]\cap D|=1$ because $A\cap D=\emptyset$. Thus,  $U\subseteq \PN(D)$ and dominates $D$. Moreover, $(\Nc_G:\mathbf{t}_{\Nt[U]\setminus D})=\langle D\rangle$ by part (a), and $|\Nt[U]\setminus D|\leq |A|\leq \mathtt{deg}(f)$.
\end{proof}

We now arrive at one of the main results of this section.
\begin{theorem}\label{t:v_number_formula}
Let $\Nc_G$ be the closed neighborhood ideal of a simple connected graph $G$. Then, 
$$
\mathtt{v}(\Nc_G) = \min\{|\Nt[U]\setminus D|\:|\: D \hbox{ is a minimal dominating set and }U\subseteq \PN(D)\hbox{ dominates } D\}.
$$
\end{theorem}
\begin{proof}
Proposition~\ref{propo:irred-decom-NG} characterizes the associated primes of the ideal $N_G$ in terms of the minimal dominating sets of $G$. Let $D$ be one of these minimal dominating sets of $G$. For a pair $(D,U)$ such that $U\subseteq \mathtt{P}_{\mathtt{N}}(D)$ and $U$ dominates $D$, Lemma~\ref{lemma:vnumber}~(a) implies that $(\Nc_G:\mathbf{t}_{\Nt[U]\setminus D})=\langle D\rangle$. From which we can conclude that 
    $$
    \mathtt{v}(\Nc_G)\leq\min\{|\Nt[U]\setminus D|\:|\: D \hbox{ is a minimal dominating set and }U\subseteq \PN(D)\hbox{ dominates } D\}.
    $$
    To prove the reverse inequality, assume that $\mathtt{v}(\Nc_G)=\mathtt{deg}(f)$, with $f\in S_d$ and  $(\Nc_G:f)=\langle D\rangle$. Lemma~\ref{lemma:vnumber}~(b) implies that there is a set  $U\subseteq \mathtt{P}_{\mathtt{N}}(D)$ that dominates $D$ and $|\Nt[U]\setminus D|\leq \mathtt{deg}(f)$. Thus, 
    $$
    \mathtt{v}(\Nc_G)\geq\min\{|\Nt[U]\setminus D|\:|\: D \hbox{ is a minimal dominating set,  and }U\subseteq \PN(D)\hbox{ dominates } D\},
    $$
    which completes the proof.
\end{proof}

\begin{remark}\label{r:minimal_c}
For a fixed minimal dominating set $D$, to compute 
$$
\min\{|\Nt[U]\setminus D|\:|\: U\subseteq \PN(D)\hbox{ dominates } D\},
$$
it is enough to consider subsets $U$ of $\PN(D)$ with $|U|=|D|$. Indeed, any $ U\subseteq \PN(D) $ that dominates $D$ must have $|U|\geq |D|$ since $|\Nt[c]\cap D|=1$ for every $c\in U$. In addition, given $U\subset U'\subseteq \PN(D)$, both dominating $D$, we obtain that $|\Nt[U]\setminus D|\leq |\Nt[\mathcal{C'}]\setminus D|$. 
\end{remark}

\begin{example}\label{ex:complete_bipartite_1}
Using Theorem \ref{t:v_number_formula}, we compute the $\vt$-number of the complete $r$-partite graph $\Nc_{K_{n_1,\dots,n_r}}$, with $r\geq 2$ and $n_1\geq \cdots \geq n_r\geq 2$. We denote the set of vertices of this graph by $V_{n_1,\dots,n_r}=\{v^1_1,\dots,v^1_{n_1},\dots,v^r_1,\dots,v^r_{n_r}\}$. All minimal dominating sets are of the form:
\begin{itemize}
\item $D_{i_k,i_t}^{k,t}=\{v_{i_k}^k,v_{i_t}^t\}$, for some $1\leq i_k \leq n_k$, $1\leq i_t\leq n_t$, $1\leq k<t \leq r$, or
\item $D(j)=\{v^j_1,\dots,v^j_{n_j}\}$, for some $1\leq j \leq r$; see Figure \ref{fig:bipartite}.
\end{itemize}
Observe the following.
\begin{itemize}
    \item We have $\PN(D_{i_k,i_t}^{k,t})=D(i_k)\cup D(i_t)\setminus D_{i_k,i_t}^{k,t}$. Consider $U\subset \PN(D_{i_k,i_t}^{k,t})$ that dominates $D_{i_k,i_t}^{k,t}$. The set $U$ must contain at least one element from $D(i_k)\setminus \{v^k_{i_k}\}$ to ensure $v^t_{i_t}\in \Nt[U]$ and at least one element from $D(i_t)\setminus \{v^t_{i_t}\}$ to ensure $v^k_{i_k}\in \Nt[U]$. Thus, $\Nt[U]=V_{n_1,\dots,n_r}$ and $|\Nt[U]\setminus D_{i_k,i_t}^{k,t}|=\sum_{j=1}^r n_j-2$. 
    \item We obtain $\PN(D(j))=D(j)$, for $1\leq j \leq r$. The only subset $U$ of $\PN(D(j))$ that dominates $D(j)$ is $U=D(j)$; for this set, $|\Nt[U]\setminus D(j)| = \sum_{i\neq j} n_i$.
\end{itemize}
As $n_1\geq \cdots \geq n_r\geq 2$, Theorem~\ref{t:v_number_formula} implies that \[\vt(\Nc_{K_{n_1,\dots,n_r}}) = \sum_{i=2}^r n_i.\] 

\begin{figure}[htbp]
    \centering
    \resizebox{0.35\textwidth}{!}{
\begin{tikzpicture}[
    vertex/.style={circle, fill=black, inner sep=2.5pt},
    label_node/.style={fill=none, inner sep=0pt},
    set_ellipse/.style={draw, thick, blue, ellipse, inner xsep=18pt, inner ysep=7pt},
    subset_ellipse/.style={draw, red, dashed, thick}
]

% D(1) Vertices
\node[vertex, label=left:{\Large $v_1^1$}] (u1) at (0, 1.5) {};
\node[vertex, label=left:{\Large $v_2^1$}] (u2) at (0, -1.5) {};

% D(2) Vertices
\node[vertex, label=right:{\Large $v_1^2$}] (v1) at (5, 3) {};
\node[vertex, label=right:{\Large $v_2^2$}] (v2) at (5, 0) {};
\node[vertex, label=right:{\Large $v_3^2$}] (v3) at (5, -3) {};

% Edges
\draw[thick] (u1) -- (v1);
\draw[thick] (u1) -- (v2);
\draw[thick] (u1) -- (v3);
\draw[thick] (u2) -- (v1);
\draw[thick] (u2) -- (v2);
\draw[thick] (u2) -- (v3);

% Ellipses for Sets D(1) and D(2)
\node[set_ellipse, fit=(u1) (u2), label=above:{\Large $D(1)$}] {};
\node[set_ellipse, fit=(v1) (v2) (v3), label=above:{\Large $D(2)$}] {};

% D^{k,t}_{i_k, i_t}
\draw[subset_ellipse, rotate around={-16.7:(2.5,-2.25)}] (2.5,-2.25) ellipse (3.6cm and 1.2cm);
\node[label_node] at (2.5, -4.1) {\Large $D^{1,2}_{2,3}$};
\end{tikzpicture}
}
    \caption{Bipartite graph and the minimal dominating sets $D(1)$, $D(2)$ and $D_{2,3}^{1,2}$.}
    \label{fig:bipartite}
\end{figure}

\end{example}

%%% section 4 Regularity and v-number 

\section{Relation between the $\vt$-number and other invariants}\label{sec:invariants}
This section is dedicated to presenting several relations between the $\vt$-number, domination number, cover number, and matching number, as well as to showing the relation with the Castelnuovo-Mumford regularity for special families of graphs.

\begin{theorem}\label{thm:cover-v}
    Let $\Nc_G$ be the closed neighborhood ideal of a simple connected graph $G$. We have
    $$\gamma(G)\leq\vt(\Nc_G)\leq \tau(G).$$
\end{theorem}

\begin{proof}	For the first inequality, let  $(D,U)$ be a pair with $D$ minimal domination set, $U\subset\PN(D)$ that dominates $D$, and $\vt(\Nc_G)=|\Nt[U]\setminus D|$. We write  $D$ as $
D=\{a_1,\dots,a_r,b_1,\dots,b_s\}$, where $\Nt(a_i)\cap \PN(D)\neq \emptyset$, for $1\leq i \leq r$, and $\Nt[b_i]\cap \PN(D)=\{b_i\}$, for $1\leq j \leq s$. 
Then, the set 
	$$
	D':=\bigcup_{j=1}^s\Nt(b_j)\cup\{a_i\:|\:a_i\in D\}
	$$
	 is a dominating set. Indeed, if $v\in V(G)$, then $v$ is adjacent to $a_i$ for some $1\leq i \leq r$ or 
     $v\in\Nt[b_j]$ for some $b_j$.   Notice that $\Nt[U]\setminus D$ contains a set of cardinality $|D'|$. Indeed, since the graph is connected, we have $|\Nt(b_j)|\geq 1$ for $1\leq j \leq s$. Since $U\subset \PN(D)$ dominates $D$ and $\Nt[b_j]\cap \PN(D)=\set{b_j}$ for $1\leq j \leq s$, we have $\set{b_1,\dots,b_s}\subset C$. Therefore,
     \[
     \bigcup_{j=1}^s\Nt(b_j)\subset \Nt[U]\setminus D.\]
     Also note that $\Nt[U]\setminus D$ contains an external private neighbor for every element  $a_i\in D$. Thus, 
    $$
    |\Nt[U]\setminus D|\geq|D'|\geq \gamma(G).
    $$
For the second inequality, let $U$ be a minimal vertex cover with $|U|=\tau(G)$. The complement of $D:=V(G)\setminus U$ is a maximal independent set, and therefore a minimal dominating set by Proposition \ref{p:maximal_independent_is_minimal_dominating}. Given that $D$ is independent, we have $D\subseteq \PN(D)$ and $D$ dominates $D$. Thus, 
$$
\tau(G)=|U|=|V(G)\setminus D|=|\Nt[D]\setminus D|\geq \vt(\Nc_G),
$$
which completes the proof.
\end{proof}

A particularly interesting case is when $\gamma(G)=\tau(G)$, since this directly determines $\vt(\Nc_G)$. These types of graphs have been studied in the literature \cite{randerath_domination_equal_covering, volkzmann_domination_equal_covering}. In particular, the minimum degree of $G$ must be either 1 or 2, and if it is 2, then $G$ must be bipartite \cite{randerath_domination_equal_covering}. We show an example with the complete bipartite graph in Example \ref{ex:regularity}. 

\begin{corollary}\label{p:bound_matching}
Let $\Nc_G$ be the closed neighborhood ideal of a simple connected graph $G$. Then,
\[\mathtt{v}(\Nc_G)\leq 2 a(G).\]
\end{corollary}
\begin{proof}
Let $M$ be a matching in $G$ with maximal size, this is $a(G)=|M|$. Notice that the set of vertices in $M$ forms a vertex cover, otherwise we can obtain a matching of size $|M|+1$. Therefore, by Theorem~\ref{thm:cover-v},
$$
\vt(\Nc_G)\leq \gamma(G)\leq2|M|=2a(G),
$$
as required.
\end{proof}

Besides the previous bound, $\mathtt{v}(\Nc_G)$ can be lower than, greater than, or equal to $a(G)$, as the next examples show. Moreover, we can get the equality $\mathtt{v}(\Nc_G)=2a(G)$ for some graphs.

\begin{example}\label{ex:matching_number}
Consider the graph $P_6$ given by a path with 6 vertices $\{v_1,\dots,v_6\}$. Then, we have that $a(P_6)=3$ and $D=\{v_2,v_5\}$ is a minimal dominating set. If we take $U=\{v_1,v_6\}$, by Theorem \ref{t:v_number_formula}, we have that $|\Nt[U]\setminus D|=2$. In other words, $\vt(\Nc_{P_6})\leq 2<a(P_6)$. One can also check that, in fact, $\vt(\Nc_{P_6})=2$.  

Now consider the complete graph $K_n$, for $n\geq 2$. Its matching number is $a(K_n)=\left\lfloor \frac{n}{2}\right\rfloor$. Let $V(K_n)=\{v_1,\dots,v_n\}$ be the vertices of $K_n$. Then any $D_i=\{v_i\}$, $1\leq i \leq n$, is a minimal dominating set. The set of private neighbors of $D_i$ is $V(K_n)$. Thus, by Theorem \ref{t:v_number_formula}, $\vt(\Nc_{K_n})=n-1$. If $n\geq 3$, we obtain $\vt(\Nc_{K_n})>a(K_n)$. In fact, we obtain
$$
\vt(\Nc_{K_n})=\begin{cases}
	2a(K_n)\hbox{, if }n\hbox{ is odd,}\\
	2a(K_n)-1\hbox{, if }n\hbox{ is even.}
	\end{cases}
$$
\end{example}

Next, we focus on the relation between the $\vt$-number and the Castelnuovo-Mumford regularity. 

\begin{example}\label{ex:regularity}
We consider the complete bipartite graph $K_{n_1,n_2}$ with $n_1\geq n_2$. From Example~\ref{ex:complete_bipartite_1}, we obtain that $\vt(\Nc_{K_{n_1,n_2}})=n_2$ (note that this can also be obtained directly from Theorem \ref{thm:cover-v}). From \cite[Thm. 2.10]{def-NG}, we have $\regt(S/\Nc_{K_{n_1,n_2}})=n_1+n_2-2$. Therefore, 
$$\regt(S/\Nc_{K_{n_1,n_2}})-\vt(\Nc_{K_{n_1,n_2}})=n_1-2.
$$
Moreover, $a(K_{n_1,n_2})=n_2$, and the difference between the regularity and the matching number can also be made arbitrarily large. 
\end{example}

For edge ideals, it has been shown that there exist families of graphs for which the $\vt$-number is smaller than the Castelnuovo–Mumford regularity. On the other hand, there are also families of graphs for which the $\vt$-number exceeds the Castelnuovo–Mumford regularity, with an arbitrarily large gap between the two invariants~\cite{civan_v_number}. In the case of closed neighborhood ideals, Example~\ref{ex:regularity} suggests that the Castelnuovo–Mumford regularity may be greater than the $\vt$-number for a broader class of graphs. As we show below, this is indeed the case. Moreover, we conjecture that this phenomenon holds in general.

\begin{theorem}[\cite{castelnuovomatch}]
For any graph $G$,
\[
\regt(S/\Nc_G) \geq a(G),
\]
 Moreover, the equality holds when $G$ is a tree.
\end{theorem}

\begin{theorem}[\cite{homologicalinvar}]\label{thm:homo-inv}
If $G$ is a bipartite graph or a very well-covered graph, then
\[
\regt(S/\Nc_G) \geq \tau(G).
\]
\end{theorem}

\begin{theorem}[\cite{homologicalinvar}]\label{thm:chordal}
If $G$ is a chordal graph, then
\[
\regt(S/\Nc_G) = \tau(G).
\]
\end{theorem}

\begin{theorem}\label{thm:vreg}
If the graph $G$ is bipartite, very well-covered, or chordal, then
$$
\vt(\Nc_G)\leq \regt(S/\Nc_G).
$$
Moreover, if $G$ is a tree, $\vt(\Nc_G)\leq\tau(G)=a(G)=\regt(S/\Nc_G).$
\end{theorem}
\begin{proof}
This result is a consequence of  Theorem~\ref{thm:cover-v}, Theorem~\ref{thm:homo-inv}, Theorem~\ref{thm:chordal} and K\"onig's theorem~\cite[Thm. 2.1.1]{diestel_graph_theory_book}. 
\end{proof}

\begin{conjecture}\label{conj:vnumber-reg}
For any simple connected graph $G$, we have
$$
\vt(\Nc_G)\leq \regt(S/\Nc_G).
$$
\end{conjecture}

\section{Relation between $\vt$-number and perfect codes}\label{sec:perfect-codes}

This section is devoted to presenting a relation between the $\vt$-number of closed neighborhood ideals and the parameters of perfect error-correcting codes. We first present the definition of an efficient domination set.

\begin{definition}[\cite{prmaximizing}]
A dominating set $D$ of graph $G$ is called an \emph{efficient dominating set} if
\[
|\Nt[v] \cap D| = 1 \quad \text{for every } v \in V(G).
\]
\end{definition}

Efficient dominating sets provide a connection between graph theory and coding theory, as the problem of determining the existence of an efficient dominating set in a graph generalizes the problem of determining the existence of a perfect code in a finite vector space, as we explain next~\cite{perfectin}.

The space $\mathbb{F}_q^m$ with the Hamming distance can be interpreted as a graph $\Gamma(m,q)$: the set of vertices corresponds to the points in $\mathbb{F}_q^m$ and there is an edge between two vertices if the Hamming distance between the corresponding points is one. In other words,  $V(\Gamma(m,q))=\mathbb{F}_q^m$ and 
$$
E(\Gamma(m,q))=\{\{u,v\}\:|\: \delta(u,v)=1\:\}.
$$

This family of graphs is called Hamming graphs and has been extensively studied from a spectral graph-theoretic perspective~\cite{hammingraphs}. When $q=2$, the Hamming graph $\Gamma(m,2)$ corresponds to the well-known hypercube graph~\cite{hypercube}.

Notice that this graph construction for codes allows us to establish the classical problem of the existence of perfect codes as a question in terms of the graph $\Gamma(m,q)$. A perfect code $C$ that corrects $t$ errors has the property that the Hamming balls $S_t(c)$ with radius $t$ and center in $c\in C$ form a partition of the space $\mathbb{F}_q^m$.  If we denote by $\Nt_t(c)$ as the set of vertices with graph distance $t$ from the vertex $c$, then a perfect code $C$ is a subset of vertices of $\Gamma(m,q)$ such that the sets $\Nt_t(c)\cup\{c\}$ form a partition of the Hamming graph. If $t=1$, then the set 
$$
\{\Nt[c]\:|\: c\in C\}
$$
forms a partition of $\Gamma(m,q)$, meaning that $C$ is an efficient domination set. In this way, an efficient dominating set corresponds to a perfect code that corrects $ 1$-errors. The notions of perfect codes, efficient dominating sets, and private neighbors in the graphs have been an active line of research considered by several 
authors; see~\cite{perfectin, privatecube, prmaximizing}.

We focus on $q$-ary Hamming codes, which are perfect codes that correct $1$ error. These codes are among the most important and classical error-correcting codes. Introduced by Richard Hamming, they were the first family of codes to provide a systematic and efficient method for detecting and correcting transmission errors. Their significance lies in their practical applications in digital communication and data storage, and they serve as a central example of perfect codes, illustrating the deep interplay among algebra, combinatorics, and information theory~\cite{hammingcodes}.

Define $n := \frac{q^r-1}{q-1}$, where $q$ is a prime power and $r \geq 2$. The $q$-ary Hamming code with redundancy $r$, denoted by $\mathcal{H}_q(r)$, is a linear code with parameters $\left[n,n-r,3\right]$ (note that the number of redundant bits is $r$). In terms of graph theory, the Hamming code $\mathcal{H}_q(r)$ is an efficient dominating set of the graph $\Gamma\left(n,q\right)$. The following example illustrates the Hamming code in the graph and the construction of $\Gamma(m,q)$.

\begin{example}\label{ex:hamming_cube_1}
    Consider the graph $\Gamma(3,2)$, which corresponds to the cube in Figure~\ref{fig:Hammingcube}. In this case, $r=2$, which means that we encode $1$ bit by using two more redundant bits. The corresponding Hamming code in this case is 
    $$
    \mathcal{H}_2(2)=\{000,111\}.
    $$
    The closed neighborhoods $\Nt[000]$ and $\Nt[111]$ form a partition of the  graph $\Gamma(3,2)$. 
    
\begin{figure}
    \centering
    \begin{tikzpicture}[scale=2]

% Front square
\node[text=red] (000) at (0,0) {$000$};
\node (100) at (1,0) {$100$};
\node[text=blue] (110) at (1,1) {$110$};
\node (010) at (0,1) {$010$};

% Back square
\node (001) at (0.4,0.4) {$001$};
\node[text=blue] (101) at (1.4,0.4) {$101$};
\node[text=red] (111) at (1.4,1.4) {$111$};
\node[text=blue] (011) at (0.4,1.4) {$011$};

% Front edges
\draw (000)--(100)--(110)--(010)--(000);

% Back edges
\draw (001)--(101)--(111)--(011)--(001);

% Connecting edges
\draw (000)--(001);
\draw (100)--(101);
\draw (110)--(111);
\draw (010)--(011);

\end{tikzpicture}
    \caption{The vertices in red compose the Hamming code $\mathcal{H}_2(2)$.}
    \label{fig:Hammingcube}
\end{figure}

\end{example}

We now use well-known facts in coding theory to compute the domination number, cover number, and independence number of the Hamming graph $\Gamma\left(n,q\right)$. We also use these results to bound the $\vt$-number of $\Gamma\left(n,q\right)$ in terms of the redundancy of the Hamming code $\mathcal{H}_q(r)$. 

\begin{lemma}\label{l:domination_hamming}
Take $n = \frac{q^r-1}{q-1}$, where $q$ is a prime power and $r \geq 2$.
We have that 
$$
\gamma\left(\Gamma\left(n,q\right)\right)=q^{n-r} \quad \text{ and } \quad
\gamma\left(\Gamma\left(m,q\right)\right)>\frac{q^m}{1+m(q-1)}
$$
if $m$ is not of the form of $\frac{q^r-1}{q-1}$.
\end{lemma}
\begin{proof}
As the Hamming code $\mathcal{H}_q(r)$ is a minimal dominating set, we get \[\gamma(\Gamma(n,q)) \leq q^{n-r}.\]
Assume we have a minimal dominating set $D$ with $|D|< q^{n-r}$. Then, $|\Nt[D]|=q^{n}$. However, as $\Nt[c]=q^r$ for any $c\in\Gamma(n,q)$, which is the size of a Hamming ball with radius 1, we have the bound $|\Nt[D]|\leq  q^r |D|<q^{n}$, which is a contradiction. Observe that the last inequality follows from the sphere packing bound for radius $1$, see Theorem \ref{t:sphere_packing}, and the non-existence of perfect codes with distance $3$ if $m$ is not of the form of $\frac{q^r-1}{q-1}$. 
\end{proof}

\begin{lemma}\label{l:hamming_cover_independence}
For any integer $m$, we have 
\[
i\left(\Gamma\left(m,q \right) \right)=q^{m-1} \text{ and }\tau\left(\Gamma\left(m,q\right)\right)=q^{m-1}(q-1).
\]
\end{lemma}
\begin{proof}
Let $L$ be the set of elements in $\mathbb{F}_q^m$ whose sum of their entries is equal to $0$. Then, $|L|=q^{m-1}$ and $L$ is a maximal independent set. Moreover, $L$ is a maximum independent set. Indeed, we can split the vertices of $\Gamma\left(m,q \right)$ into $q^{m-1}$ disjoint sets by fixing their first $m-1$ entries. The vertices in each of those sets are connected, since they are at Hamming distance 1 apart. Therefore, a maximal independent set can have at most $1$ vertex from each of those sets, that is, its size is bounded by $q^{m-1}$. 

Since $L$ is independent if and only if $V\setminus L$ is a vertex cover, we have $i\left(\Gamma\left(m,q \right) \right)+ \tau\left(\Gamma\left(m,q\right)\right)=q^m$, which gives the result.
\end{proof}

\begin{theorem}\label{thm:Hammingdomination}
Take $n = \frac{q^r-1}{q-1}$, where $q$ is a prime power and $r \geq 2$. Then, 
$$
q^{n-r}\leq \vt(\Nc_{\Gamma(n,q)})\leq q^{n-1}(q-1).
$$ 
\end{theorem}

\begin{proof}
It follows from Theorem \ref{thm:cover-v} and Lemmas \ref{l:domination_hamming} and \ref{l:hamming_cover_independence}. 
\end{proof}

We use the next two examples to illustrate that the previous upper bound is sharp.  
\begin{example}\label{ex:hamming_cube_2}
Following the setting from Example~\ref{ex:hamming_cube_1}, using Theorem~\ref{thm:Hammingdomination}, we get
    $$
    2 \leq \vt(\Nc_{\Gamma(3,2)})\leq 4.
    $$ 
Using the code from Appendix \ref{sec:computer_prog}, one can check that the upper bound is sharp in this case, i.e., $\vt(\Nc_{\Gamma(3,2)})=4$. 
\end{example}

\begin{example}
The well-known Hamming code $\mathcal{H}_2(3)$ is an efficient dominating set of the graph $\Gamma(7,2)$~\cite{hammingcodes}. The code $\mathcal{H}_2(3)$ uses $3$ redundant bits to encode $4$ bits, this means it has parameters $[7,4,3]$. The Hamming code can be represented by the following codewords
\begin{align*}
\mathcal{H}_2(3)=\{
&0000000,\,
0001011,\,
0010111,\,
0011100,\,
0100110,\,
0101101,\,
0110001,\,
0111010,\,\\
&1000101,\,
1001110,\,
1010010,\,
1011001,\,
1100011,\,
1101000,\,
1110100,\,
1111111
\}.    
\end{align*}
In Figure~\ref{fig:hammingcode}, we display part of the graph $\Gamma(7,2)$ that represents two close neighborhoods of two of the elements in $\mathcal{H}_2(3)$. Using Theorem~\ref{thm:Hammingdomination}, we obtain that 
$$
16=2^4\leq\vt(\Nc_{\Gamma(7,2)})\leq 2^{6}=64.
$$
 
\begin{figure}
    \centering
    \begin{tikzpicture}[scale=1.6]

%--------------------------------
% First neighborhood: top left
%--------------------------------
\node[text=red] (c1) at (-3,2) {$0000000$};

\node (a1) at (-1.8,2) {$1000000$};
\node (a2) at (-2.1,2.9) {$0001000$};
\node (a3) at (-3,3.2) {$0100000$};
\node (a4) at (-3.9,2.9) {$0010000$};
\node (a5) at (-4.2,2) {$0000100$};
\node (a6) at (-3.9,1.1) {$0000010$};
\node (a7) at (-3,0.8) {$0000001$};

\draw (c1)--(a1);
\draw (c1)--(a2);
\draw (c1)--(a3);
\draw (c1)--(a4);
\draw (c1)--(a5);
\draw (c1)--(a6);
\draw (c1)--(a7);

%--------------------------------
% Second neighborhood: top right
%--------------------------------
\node[text=red] (c2) at (3,2) {$0001011$};

\node (b1) at (4.2,2) {$1001011$};
\node (b2) at (3.9,2.9) {$0011011$};
\node (b3) at (3,3.2) {$0101011$};
\node (b4) at (2.1,2.9) {$0000011$};
\node (b5) at (1.8,2) {$0001001$};
\node (b6) at (2.1,1.1) {$0001010$};
\node (b7) at (3,0.8) {$0001111$};

\draw (c2)--(b1);
\draw (c2)--(b2);
\draw (c2)--(b3);
\draw (c2)--(b4);
\draw (c2)--(b5);
\draw (c2)--(b6);
\draw (c2)--(b7);

%--------------------------------
% Connecting edges between neighborhoods
%--------------------------------
\draw (a2)--(b5); % 0010000 -- 0010010
\draw (a2)--(b6); % 0010000 -- 0010010
\draw (b5)--(a7);
\draw (b4)--(a7);

\end{tikzpicture}
    \caption{Two close neighborhoods of elements in $\mathcal{H}_2(3)$ with some edges between the two neighborhoods.}
    \label{fig:hammingcode}
\end{figure}

\end{example}

% \bibliographystyle{alpha}
% \bibliography{References}
\newcommand{\etalchar}[1]{$^{#1}$}

\appendix
\section{Code for examples}\label{sec:computer_prog}
In this section, we provide the Sage code for the computations that appear in this paper~\cite{sagemath}. For the computation of the $\vt$-number using Theorem \ref{t:v_number_formula}, we have taken into account Remark \ref{r:minimal_c} to make the computation more manageable. We first provide the basic functions.

\begin{minted}[frame=single, breaklines]{python}
def closed_neighborhood(G, S):
    """Returns the closed neighborhood N[S] of a set of vertices S."""
    N_S = set(S)
    for v in S:
        N_S.update(G.neighbors(v))
    return N_S
def CNI(G, base_ring=QQ):
    """Constructs the closed neighborhood ideal of a graph G."""
    V = G.vertices(sort=True)
    var_names = ['x_{}'.format(i) for i in range(len(V))]
    
    R = PolynomialRing(base_ring, var_names)
    gens = R.gens()
    
    var_map = {V[i]: gens[i] for i in range(len(V))}
    
    ideal_generators = []
    
    for v in V:
        c_neighborhood = closed_neighborhood(G, {v})
        monomial = R(1)
        for u in c_neighborhood:
            monomial *= var_map[u]
        ideal_generators.append(monomial)
    return R.ideal(ideal_generators)
def v_number(I):
    """Computes the v-number of a monomial ideal I."""
    ass_primes = I.associated_primes()
    v_num = infinity
    
    for P in ass_primes:
        colon_ideal = I.quotient(P)
        
        # Check the minimal generators of the colon ideal
        for f in colon_ideal.interreduced_basis():
            if I.quotient(ideal(f)) == P:
                d = f.degree()
                if d < v_num:
                    v_num = d          
    return v_num
def v_number_formula(G):
    """ Computes the v-number using the formula from the paper."""
    # Precompute all closed neighborhoods
    V = set(G.vertices())
    N_cache = {v: closed_neighborhood(G, {v}) for v in V}
    
    min_val = infinity
    best_D = None
    best_U = None
    
    for D in G.minimal_dominating_sets():
        # Group private neighbors in terms of vertices of D
        PN_dict = {d: [] for d in D}
        
        for v in V:
            intersect = N_cache[v].intersection(D)
            if len(intersect) == 1:
                # v is a private neighbor
                d = next(iter(intersect)) # extract the only element in the intersection
                PN_dict[d].append(v)
            
        for U_tuple in itertools.product(*PN_dict.values()): # Requires itertools
            
            # Compute N[U]
            N_U = set()
            for u in U_tuple:
                N_U.update(N_cache[u])
                
            current_val = len(N_U) - len(D)
            
            if current_val < min_val:
                min_val = current_val
                best_D = D
                best_U = set(U_tuple)
                
    return min_val, best_D, best_U
\end{minted}

Now we provide the code to obtain the results from Example \ref{ex:regularity}.

\begin{minted}[frame=single, breaklines]{python}
n1 = 5
n2 = 3 #n1>=n2>=2

import itertools
from sage.all import singular
singular.lib('mregular.lib')

G = Graph()
G = graphs.CompleteBipartiteGraph(n1,n2)

# Generate the ideal
NI = CNI(G)

matching_edges = G.matching()
matching_number = len(matching_edges)

v_cover = G.vertex_cover()
v_cover_size = len(v_cover)
vn=v_number(NI)
reg=singular.regIdeal(NI)-1

print('\nMatching number: ', matching_number, 'Minimal vertex cover: ', v_cover_size, 'v-number: ', vn,'Regularity: ',reg)
val, D_min, U_min = v_number_formula(G)
print('Formula v-number: ', val, 'D_min: ',D_min, 'U_min: ', U_min, 'N(U) :', closed_neighborhood(G, U_min))     
\end{minted}

\end{document}